\documentclass[UTF-8,reqno]{amsart}
\usepackage{enumerate, bbm}
\usepackage{amssymb,url,color, booktabs}
\usepackage{tikz}
\usepackage{mathrsfs}
\usepackage{dutchcal}
\usepackage{threeparttable}
\usepackage{color}
\usepackage[colorlinks=true]{hyperref}
\hypersetup{
    linkcolor=blue,          
    citecolor=red,        
    filecolor=blue,      
    urlcolor=cyan
}

\usepackage{color}
\usepackage{ulem}

 \usepackage{scalerel} 

\definecolor{MyDarkBlue}{cmyk}{0.8,0.3,0.8,0.4}
\definecolor{yellow}{rgb}{0.99,0.99,0.70}
\definecolor{white}{rgb}{1.0,1.0,1.0}
\definecolor{black}{rgb}{0.00,0.00,0.00}

\numberwithin{equation}{section}

\newcommand{\be}{\begin{eqnarray}}
\newcommand{\ee}{\end{eqnarray}}
\newcommand{\ce}{\begin{eqnarray*}}
\newcommand{\de}{\end{eqnarray*}}
\newtheorem{theorem}{Theorem}[section]
\newtheorem{lemma}[theorem]{Lemma}
\newtheorem{remark}[theorem]{Remark}
\newtheorem{definition}[theorem]{Definition}
\newtheorem{proposition}[theorem]{Proposition}
\newtheorem{Examples}[theorem]{Example}
\newtheorem{corollary}[theorem]{Corollary}

\usepackage[nobysame]{amsrefs}
\BibSpec{article}{%
+{}{\PrintAuthors} {author}
+{,}{ \textrm} {title}
+{.}{ \textit} {journal}
+{,}{ \textbf} {volume}
+{}{ \parenthesize} {date}
+{,}{ } {pages}
+{.}{ arXiv:} {eprint}
+{.}{} {transition}
}
\BibSpec{book}{%
+{}{\PrintAuthors} {author}
+{,}{ \textit} {title}
+{.}{ \textrm} {series} 
+{,}{ Vol.} {volume} 
+{.}{ } {publisher}
+{,}{ } {date}
+{.}{} {transition}
}

\def\var{{\mathrm{var}}}

\def\eps{\varepsilon}

\def\p{\partial}

\def\[{{\Big[}}
\def\]{{\Big]}}
\def\<{{\langle}}
\def\>{{\rangle}}
\def\({{\Big(}}
\def\){{\Big)}}

\def\bx{{\mathbf{x}}}

\def\dif{{\mathord{{\rm d}}}}

\def\no{\nonumber}
\def\={&\!\!=\!\!&}

\def\mE{{\mathbb E}}

\def\mN{{\mathbb N}}

\def\mP{{\mathbb P}}

\def\mR{{\mathbb R}}

\def\1{{\mathbf{1}}}

\def\sI{{\mathscr I}}

\def\geq{\geqslant}
\def\leq{\leqslant}
\def\ge{\geqslant}
\def\le{\leqslant}

\def\div{\mathord{{\rm div}}}

\def\var{{\mathrm{var}}}

\def\eps{\varepsilon}

\def\p{\partial}

\def\[{{\Big[}}
\def\]{{\Big]}}
\def\<{{\langle}}
\def\>{{\rangle}}
\def\({{\Big(}}
\def\){{\Big)}}

\def\bx{{\mathbf{x}}}

\def\dif{{\mathord{{\rm d}}}}

\def\no{\nonumber}
\def\={&\!\!=\!\!&}
\def\bt{\begin{theorem}}
\def\et{\end{theorem}}
\def\bl{\begin{lemma}}
\def\el{\end{lemma}}
\def\br{\begin{remark}}
\def\er{\end{remark}}
\def\bx{\begin{Examples}}
\def\ex{\end{Examples}}
\def\bd{\begin{definition}}
\def\ed{\end{definition}}
\def\bp{\begin{proposition}}
\def\ep{\end{proposition}}
\def\bc{\begin{corollary}}
\def\ec{\end{corollary}}

\def\geq{\geqslant}
\def\leq{\leqslant}
\def\ge{\geqslant}
\def\le{\leqslant}

\def\div{\mathord{{\rm div}}}

\def\<{\langle} \def\>{\rangle}

\def\bpf{\begin{proof}}
\def\epf{\end{proof}}

\allowdisplaybreaks

\begin{document}
	\title[Convergence rate of Euler scheme to dDSDE]{Convergence rate of the Euler-Maruyama scheme to density dependent SDEs driven by $\alpha$-stable additive noise}



\author{Ke Song}
\address[Ke Song]{Department of Mathematics, Beijing Institute of Technology, Beijing 100081, China}
\curraddr{}
\email{ske2022@126.com}
\thanks{K. Song is grateful to the financial supports by National Key R \& D Program of China (No. 2022YFA1006300) and the financial supports of the NSFC (No. 12271030).}


\author{Zimo Hao}
\address[Zimo Hao]{Fakult\"at f\"ur Mathematik, Universit\"at Bielefeld,
	33615, Bielefeld, Germany}
\curraddr{}
\email{zhao@math.uni-bielefeld.de}
\thanks{Z. Hao is supported by DFG through the CRC 1283/2 2021 - 317210226
``Taming uncertainty and profiting from randomness and low regularity in analysis, stochastics and their applications''. }


\date{}

\dedicatory{}

\commby{}

\begin{abstract} In this paper, we establish the weak convergence rate of density-dependent stochastic differential equations with bounded drift driven by $\alpha$-stable processes with $\alpha\in(1,2)$. The well-posedness of these equations has been previously obtained in \cite{wu2023well}. We derive an explicit convergence rate in total variation for the Euler-Maruyama scheme, employing a technique rooted in \cite{hao2023}.
\end{abstract} 

\maketitle

\tableofcontents
\section{Introduction}
In this paper, we let $\alpha \in (1,2)$ and consider the Euler-Maruyama scheme applied to the following density-dependent stochastic differential equation (dDSDE):
\begin{equation} \label{1}
	\dif X_t=b(t,X_t,\rho_t(X_t)) \dif t+\dif L_t, \quad X_{0} \stackrel{(d)}{=} \mu_{0},
\end{equation}
where $(L_t)_{t \ge 0}$ is a $d$-dimensional symmetric and rotationally invariant $\alpha$-stable process on some probability space $(\Omega, \mathcal{F}, \mathbb{P})$, $b:\mathbb{R}_+ \times \mathbb{R}^d \times \mathbb{R}_+ \to \mathbb{R}^d$ is a bounded Borel measurable vector field, $\mu_{0}$ is a probability measure over $\mathbb{R}^d$ and for $t>0, \rho_t(x)=\mathbb{P} \circ X_t^{-1}(\dif x)/\dif x$ is the distributional density of $X_t$ with respect to (w.r.t.) the Lebesgue measure $\dif x$ on $\mathbb{R}^d.$
By It\^o's formula, one sees that $\rho_t$ solves the following nonlinear Fokker-Planck equation (FPE) in the distributional sense:
 \begin{align}\label{P13}
\p_t\rho_t-\Delta^{\alpha/2}\rho_t+\div(b(t,\cdot,\rho_t)\rho_t)=0,\quad \lim_{t\downarrow 0}\rho_t=\mu_0\mbox{ weakly},
\end{align}
where the fractional Laplacian operator, which is the infinitesimal generator of $\alpha$-stable process $\left(L_{t}\right)_{t \geqslant 0}$, is defined by  
	$$
	\Delta^{\alpha / 2}:=c\int_{\mathbb{R}^d} \left( \phi(x+z)-\phi(x)-z \mathbf{1}_{|z| \le 1} \cdot \nabla \phi(x)  \right) |z|^{-d-\alpha} \dif z,
	$$
	with some specific constant $c=c(d, \alpha)>0$.
More precisely, for any $\varphi\in C_0^\infty(\mR^d)$,
\begin{align}\label{DPDE}
\<\rho_t, \varphi\>=\<\mu_0,\varphi\>+\int_0^t\<\rho_s,\Delta^{\alpha/2}\varphi\>\dif s
+\int_0^t\<\rho_s, b(s,\cdot,\rho_s)\cdot\nabla\varphi\>\dif s,
\end{align}
where $\<\rho_t, \varphi\>:=\int_{\mR^d}\varphi(x)\rho_t(x)\dif x=\mE\varphi(X_t)$.

In recent years, the study of distributional dependent stochastic differential equations (DDSDEs), also known as McKean-Vlasov SDEs, has garnered significant attention due to their wide range of applications. These applications span various fields, including mean-field games (see e.g. \cite{CD18}), vortex models (see e.g. \cite{Se20}), and kinetic theory (see e.g. \cite{Ka56}). 
The general form of such an equation is given by:
\begin{align*}
\dif X_t=B(t,X_t,\mu_t)\dif t+\Sigma(t,X_t,\mu_t)\dif W_t,
\end{align*}
where $\mu_t$ is the time marginal distribution of the solution. 
Considerable research has been devoted to various aspects of these equations, including well-posedness (see \cite{MV16, RZ21} for instance), derivative formulas (see \cite{RW19} for instance), long-time behavior (see \cite{HRW21, HRZ24b} for instance), and the propagation of chaos of $N$-particle systems (see \cite{Sz91, JW18, La21} for instance) and others related topics. These studies typically assume that the coefficients $B$ and $\Sigma$ depend continuously on $\mu$ in terms of Wasserstein or the total variation distances.


In this paper, we focus on a specific case where the drift term $B(t,x,\mu)$ depends on the distributional density, denoted by $b(t,x,\frac{\dif \mu}{\dif x}(x))$. This type of dependence is referred to as ``Nemytskii-type" and represents a significant departure from the traditional McKean-Vlasov SDEs. Such dDSDEs (also called McKean-Vlasov SDEs of Nemytskii-type) were first introduced in Section 2 of Barbu and R\"ockner's work \cite{BR20}. Unlike classical McKean-Vlasov SDEs, the mapping $\mu\to b(t,x,\frac{\dif \mu}{\dif x}(x))$ is even not continuous with respect to the weak convergence topology. This discontinuity presents additional challenges in the analysis of these equations. 

Despite these challenges, the well-posedness of dDSDEs driven by both Brownian motion and L\'evy processes has been established in a series of papers by Barbu and R\"ockner \cite{BR20, BR21a, BR21b, BR23a, BR23b, BR24}. Further literature on the well-posedness of these equations with singular drift $b$ can be found in recent works \cite{Wa23, HRZ24, Le24}. 



In this paper, we consistently assume the following condition is upheld: 
\begin{itemize}
	\item[{\textbf{(H)}}]
	There is a constant $\kappa>0$ such that for all $(t,x,u_i) \in \mathbb{R} \times \mathbb{R}^d \times \mathbb{R}_+, i=1,2,$ 
	\begin{equation}\label{6}
		|b(t,x,u_1)|\le \kappa\qquad \text{and}\quad \qquad |b(t,x,u_1)-b(t,x,u_2)| \le \kappa|u_1-u_2|.
	\end{equation}
	Moreover, $\mu_0(\dif x)=\rho_0(x)\dif x$ with $\rho_0\in L^q(\mR^d)$ for some $q\in (\frac{d}{\alpha-1}\vee2,\infty]$. Here $L^q(\mR^d)$, $q\in[1,\infty]$, is the space of all Lebesgue measurable and $q$-integrable functions on $\mR^d$, with the standard norm $\|\cdot\|_{q}$.  
\end{itemize}
Under the condition {\textbf{(H)}}, a unique weak solution to dDSDE \eqref{1} on $[0, T]$ was obtained for arbitrary $T>1$ in \cite{wu2023well}
in the sense of the following definition.

\begin{definition}(Weak Solutions) \label{3}
	Let $\mu_{0}$ be a probability measure on $\mathbb{R}^d$ and $\alpha \in (1,2).$ We call a filtered probability space $(\Omega, \mathcal{F}, \mathbb{P};(\mathcal{F}_t)_{t \ge 0})$ together with a pair of $(\mathcal{F}_t)$-adapted processes $(X_t,L_t)_{t \ge 0}$ defined on it a weak solution of SDE \eqref{1} with initial distribution $\mu_{0}$, if 
	\begin{enumerate}
		\item[(i)] $ \mathbb{P} \circ X_0^{-1} = \mu_{0},$ and $(L_t)_{t \ge 0}$ is a $d$-dimensional symmetric and rotationally invariant $\alpha$-stable process.
		\item[(ii)] for each $t>0,$ $\rho_t(x)=\mathbb{P} \circ X_t^{-1} (\dif x)/ \dif x$ and 
		\begin{equation*}
			X_t=X_0+\int_{0}^{t} b(s,X_s, \rho_s(X_s)) \dif s+L_t, \quad \mathbb{P} \text{-a.s.}
		\end{equation*}
	\end{enumerate}
\end{definition}
Specifically, the weak solution obtained in \cite{wu2023well} is constructed by the limit of the following Euler-Maruyama scheme:
Letting $T>1$ and $h \in (0,1)$, define
\begin{equation} \label{Euler1}
	X_t^h:=X_0+ L_t, \quad \text{$t \in [0,h],$}
\end{equation}				
and for $t \in [kh,(k+1)h]$ with $k=1,2,\ldots,\left[\frac{T}{h} \right],$
\begin{equation} \label{Euler2}
	X_t^h=X_{kh}^h+\int_{kh}^{t} b^h(s,X_{kh}^h, \rho_{kh}^h(X_{kh}^h)) \dif s+(L_t-L_{kh}),
\end{equation}
where $\rho_{t}^h$ is the distributional density of $X_{t}^h$ w.r.t. the Lebesgue measure. Its existence is obvious from the construction of $X_{t}^h$.  
Denoting
\begin{equation*}
	\pi_h(s):=\sum_{j=0}^{\infty} jh \mathbf{1}_{[jh,(j+1)h)} (s),
\end{equation*}
from the construction \eqref{Euler1} and \eqref{Euler2}, it is evident that $X_t^h$ solves the following SDE:
\begin{equation*}
	X_t^h=X_0+\int_{0}^{t} b^h(s,X_{\pi_h(s)}^h) \dif s + L_t,
\end{equation*}
where
\begin{equation*}
	b^h(s,x):=\mathbf{1}_{\{s \ge h \}} b \left(s,x,\rho_{\pi_h(s)}^h (x)\right).
\end{equation*}
In \cite{wu2023well}, Wu and the second named author provided uniform estimates for ${\rho^h, h>0}$ in $h$. However, they did not provide a quantitative estimate of the difference $\rho^h-\rho$, which is crucial in numerous areas such as numerical analysis, mathematical finance, and stochastic control theory. 
Even for the case $b=b(t,x)$, which does not depend on the distribution, except for a recent work \cite{FJM24}, there seems no result concerning the convergence rate of Euler approximation when $b$ is only bounded (we will discuss the related results later). 
Therefore, in this paper, we investigate the convergence rate of this Euler-Maruyama approximation.

\subsection{Main results}
The following is the main result of this article.
\begin{theorem}\label{11}
	Assume {\textbf{(H)}} holds. 
	Then there is a constant $C=C(d,T, \kappa,q,\|\rho_0\|_{q})$ such that for all $t \in [0,T]$ and $h \in (0,1)$,
	\begin{equation}\label{0425:00}
		\| \rho_t-\rho_t^h\|_1 \le Ch^{\frac{\alpha-1}{\alpha}}.
	\end{equation}
\end{theorem}

%


Notably, there appears to be a scarcity of results addressing the Euler-Maruyama approximation to the dDSDE \eqref{1}. 
For the case of Brownian motion, the convergence of the Euler-Maruyama approximation has been obtained in \cite{HRZ21}, with a convergence rate in the total variation of $1/2$ established in \cite{hao2023}. Similarly, for $\alpha$-stable cases, the well-posedness has been addressed in \cite{wu2023well}.

Regarding our main results, even in the case where $b(t,x,u)=b(t,x)$ is not density-dependent, Theorem \ref{11} represents a novel contribution. Importantly, the convergence rate therein does not depend on the continuous regularity of the drift term. However, it is worth noting that in comparison to the strong convergence result established in \cite{BDG22}, their results do not directly yield total variation convergence from path convergence. Furthermore, the convergence rate obtained in \cite{FJM24} is $(\alpha-1)/\alpha-\eps$ for any small $\eps>0$ when $b$ is not density-dependent and bounded. In comparison, our results provide a clear rate without the $\eps$. Nonetheless, the optimality of this rate for our case remains uncertain and may warrant further investigation to determine the precise optimal rate.

\subsection{The primary outline of the proof} In this section, we provide an outline of the proof for Theorem \ref{11}. The complete details can be found in Section \ref{sec3}. To show our main result, we note that $\| \rho_t-\rho_t^h\|_1$ is equivalent to the total variation distance:
\begin{align*}
\|\mP\circ(X_t)^{-1}-\mP\circ(X_t^h)^{-1}\|_{\rm var}&:=\sup_{\phi\in L^\infty(\mR^d);\|\phi\|_\infty=1}|\mE\phi(X_t)-\mE\phi(X_t^h)|\\
&=\sup_{\phi\in C^\infty_b(\mR^d);\|\phi\|_\infty=1}|\mE\phi(X_t)-\mE\phi(X_t^h)|.
\end{align*}
Then we consider the function $u^t(s,x):=\mE \phi(x+L_{t-s}),$ which solves the following backward heat equation:
\begin{equation} 
	\partial_s u^t+ \Delta^{\frac{\alpha}{2}} u^t=0, \quad u^t(t,x)=\phi(x), \quad s \in (0,t), \ x \in \mathbb{R}^d.
\end{equation}
 By applying It\^o's formula, we obtain the following relationship:
\begin{align}\label{0527:00}
\mathbb{E} \phi(X_t^h)-\mathbb{E} \phi(X_t)=\mathbb{E} \int_{0}^{t} b^h(s, X_{\pi_h(s)}^h) \cdot \nabla u^t(s, X_{s}^h) \dif s-\mathbb{E} \int_{0}^{t} (b^0\cdot \nabla u^t)(s,X_{s}) \dif s,
\end{align}
where $b^0(t,x):=b(t,x,\rho_t(x))$. This representation, known as the It\^o-Tanaka trick, has found extensive application across diverse academic domains (see \cite{MPT17} for Euler approximation, \cite{RSX17, CHR24} for slow-fast systems, and \cite{WH23} for well-posedness). Our contribution lies in using the regularity of the semigroup $\phi\to \mathbb{E}[\phi(x+L_{t})]$, drawing inspiration from interpolation techniques to achieve the convergence rate of $(\alpha-1)/\alpha$.

 More precisely, in estimating the right-hand side of \eqref{0527:00},
 a crucial term is
\begin{align*}
I^h:=\mathbb{E} \int_{0}^{t} \left(b^h(s,X_{\pi_h(s)}^h) -b^h( s, X_{s}^h) \right) \cdot \nabla u^t(s,X_{s}^h) \dif s.
\end{align*}
 To estimate it, 
 we make the decomposition as follows:
 \begin{align*}
I^h=&\mathbb{E} \int_{0}^{t}(b^h\cdot \nabla u^t)(s,X_{\pi_h(s)}^h) -(b^h\cdot \nabla u^t)( s, X_{s}^h)  \dif s\\
&+\mathbb{E} \int_{0}^{t} b^h(s,X_{\pi_h(s)}^h)\cdot\left(\nabla u^t(s,X_{s}^h) -\nabla u^t( s, X_{\pi_h(s)}^h) \right) \dif s,
\end{align*}
 where the first one is bounded by $\int_0^t\|\rho^h_{\pi_h(s)}-\rho^h_s\|_1\|b^h\cdot\nabla u^t(s)\|_\infty\dif s$, which has a convergence rate of $h^{(\alpha-1)/\alpha}$ as shown in Lemma \ref{lem:22}. To estimate the second term, we employ an interpolation approach by dividing the integral over the interval $s \in [0, t]$ into two parts: $s \in [t-h, t]$ and $s \in [0, t-h]$. For $s \in [t-h, t]$, we control the term using $\|\nabla^t u(s)\|_\infty \lesssim (t-s)^{-1/\alpha}$, and obtain the convergence rate from the smallness of the integral interval. For $s \in [0, t-h]$, we apply the result established in Lemma \ref{lem:0425}, which crucially depends on the heat kernel estimate from \cite{wu2023well}, to control the term by $h(t-s)^{-(1+\alpha)/\alpha}$. This control includes an additional rate factor $h$ and an integrable function $s \mapsto (t-s)^{-(1+\alpha)/\alpha}$ over $[0, t-h]$. Further details can be found in the estimate for the term $\sI_{42}$ in Section \ref{sec3}.
 



\subsection{
Relevant literature}
The convergence properties of the Euler-Maruyama scheme, when applied to SDEs driven by both Brownian motion and L\'evy processes with smooth coefficients, have undergone extensive scrutiny and are thoroughly expounded upon in the literature. Recent research has been particularly focused on investigating the strong (path) convergence rates for SDEs with irregular coefficients, such as bounded and $L^p$ drift terms. This renewed interest has been motivated by insights into the regularity properties of the diffusion semigroup.


In the case of SDEs driven by additive Brownian motion noise, Dareiotis and Gerencs\'er in \cite{DG20} provided an $L^2$ moment rate of $1/2-$ (interpreted as $1/2-\eps$ for any small $\eps>0$) for any Dini-continuous drift. Meanwhile, for SDEs with bounded measurable drift, the total variation rate of convergence was established as $1/2$ in \cite{BJ20}, coinciding with the rate derived in our work for $\alpha=2$. Additionally, as the stochastic sewing lemma, introduced by L\^e in \cite{Le20}, has gained popularity for analyzing the convergence rate of the Euler-Maruyama scheme with singular coefficients, an array of developments relying on it appeared, including the study of SDEs with Sobolev coefficients \cite{DGL23}, $L^p$ $(p>d)$ drifts \cite{LL24}, and those driven by fractional Brownian motion \cite{BDG21}. We also direct readers to the references provided therein for relevant literature.


Regarding the $\alpha$-stable cases, Mikulevi$\check{\rm c}$ius and Xu in \cite{MX20} obtained the strong convergence rate under $\alpha\in(1,2)$ and $\beta$-H\"older drifts with $\beta>1-\alpha/2$. This result was extended to the case $\alpha<1$ in \cite{LZ23} (see also \cite{MPT17, HL18, KS19} for related results). Notably, the convergence rates derived in these studies rely upon the regularity of the drift coefficient and deteriorate, approaching zero, as 
$\beta$ tends to zero. In contrast, using the stochastic sewing technique, Butkovsky, Dareiotis, and Gerencs\'er established a strong convergence rate of $1/2+$ for SDEs driven by additive $\alpha$-stable processes with $\beta$-H\"older continuous drifts, as detailed in \cite{BDG22}. When $b$ is independent of the distributional density $\rho_t(x)$, we also refer to the very recent article \cite{FJM24} on taming Euler scheme with $L^q-L^p$ drift, which achieves a convergence rate of order $\frac{1}{\alpha} \left(\alpha-1-\left(\frac{\alpha}{q}+\frac{d}{p} \right) \right),$ for the difference of the densities.

\subsection{Conventions and notations}
We close this section by introducing the following conventions and notations used throughout this paper: As usual, we use $:=$ as a way of definition. Define $\mN_0:= \mN \cup \{0\}$ and $\mR_+:=[0,\infty)$. The letter $c=c(\cdots)$ denotes an unimportant constant, whose value may change in different places. We use $A \asymp B$ and $A\lesssim B$ to denote $c^{-1} B \leq A \leq c B$ and $A \leq cB$, respectively, for some unimportant constant $c \geq 1$.  We also use $A   \lesssim_c  B$ to denote $A \leq c B$ when we want to emphasise the constant.
	

	\section{Preliminaries}
	\subsection{$\alpha$-stable processes}
	A c\`adl\`ag process $\{ L_t; t \ge 0\}$ on $\mathbb{R}^d$ is called a L\'evy process, if $L_0=0$ almost surely and $L$ has independent and identically distributed increments. The associated Poisson random measure is defined by
	\begin{equation*}
		N((0,t] \times \Gamma):=\sum_{s \in (0,t]} \mathbf{1}_{\Gamma} (L_s-L_{s-}), \ \Gamma \in \mathcal{B}(\mathbb{R}^d / \{0\}), \  t>0,
	\end{equation*}
	and the L\'evy measure is given by 
	\begin{equation*}
		\nu(\Gamma) :=\mathbb{E}N((0,1] \times \Gamma).
	\end{equation*}
	Then, the compensated Poisson random measure is defined by
	\begin{equation*}
		\widetilde{N}(\dif r, \dif z) :=N(\dif r, \dif z)-\nu(\dif z) \dif r.
	\end{equation*}
	For $\alpha \in (0,2),$ a L\'evy process $L_t$ is called a symmetric and rotationally invariant $\alpha$-stable process if the L\'evy measure has the form 
	\begin{equation*}
		\nu^{(\alpha)} (\dif z)= |z|^{-d-\alpha} \dif z.
	\end{equation*}
	In this paper, we only consider the symmetric and rotationally invariant $\alpha$-stable process. Without confusing, we simply call it the $\alpha$-stable process. For any $0 \le \gamma_1 < \alpha < \gamma_2,$ it is easy to see that 
	\begin{equation*}
		\int_{\mathbb{R}^d} (|z|^{\gamma_1} \wedge |z|^{\gamma_2}) \nu^{(\alpha)}(\dif z) < \infty.
	\end{equation*}
	Based on \cite[Proposition 28.1]{ken1999levy}, $L_t$ admits a smooth density function $q_{\alpha}(t, \cdot)$ given by Fourier's inverse transform
	\begin{equation*}
		q_{\alpha}(t,x)=(2 \pi)^{-d/2} \int_{\mathbb{R}^d} e^{-ix \cdot \xi} \mathbb{E} e^{i \xi \cdot L_t} \dif \xi, \quad \forall t >0.
	\end{equation*}
	Since the $\alpha$-stable process $L_{t}$ has the scaling property
	\begin{equation*}
		\left(\lambda^{-1 / \alpha} L_{\lambda t}\right)_{t \geqslant 0} \stackrel{(d)}{=}\left(L_{t}\right)_{t \geqslant 0}, \quad \forall \lambda>0,
	\end{equation*}
	it is easy to see that
	\begin{equation}\label{0422:00}
		q_{\alpha}(t, x)=t^{-d / \alpha} q_{\alpha}\left(1, t^{-1 / \alpha} x\right) .
	\end{equation}
	By \cite[Theorem 2.1]{blumenthal1960some}, one knows that there is a constant $c=c(d, \alpha)>1$ such that
	\begin{equation}\label{0423:00}
		c^{-1} \varrho_{\alpha}(t, x) \leqslant q_{\alpha}(t, x) \leqslant c \varrho_{\alpha}(t, x),
	\end{equation}
	where 
	\begin{equation*}
		\varrho_{\alpha}(t, x):=\frac{t}{\left(t^{1 / \alpha}+|x|\right)^{d+\alpha}} .
	\end{equation*}
	By \cite[Lemma 2.2]{chen2016heat}, for any $j \in \mathbb{N}_{0}$, there is a constant $c=c(d, j, \alpha)>0$ such that
	\begin{equation} \label{nabla-p}
		\left|\nabla^{j} q_{\alpha}(t, x)\right| \leqslant c t^{-j / \alpha} \varrho_{\alpha}(t, x) .
	\end{equation}
	Thus, in view of \eqref{0422:00}, we have for any $j\in \mathbb{N}_{0}$ and $q\in[1,\infty]$,
	\begin{equation}\label{estimate-varro}
		\| \nabla^j q_\alpha (t) \|_{L^{q/{q-1}}}= t^{-\frac{d}{\alpha q}-\frac{j}{\alpha}}\|\nabla^j q_\alpha (1)\|_{L^{q/(q-1)}}\lesssim t^{-\frac{d}{\alpha q}-\frac{j}{\alpha}}.
	\end{equation}
	Note that $q_{\alpha}(t, x)$ is the heat kernel of the operator $\Delta^{\alpha / 2}$, i.e.,
	\begin{equation}\label{0423:02}
		\partial_{t} q_{\alpha}(t, x)=\Delta^{\alpha / 2} q_{\alpha}(t, x), \qquad \lim _{t \downarrow 0} q_{\alpha}(t, x)=\delta_{0}(x),
	\end{equation}
	where $\delta_{0}$ is the Dirac measure at point $ 0 $. We also have the following Chapman-Kolmogorov 
	equations:
	\begin{equation} \label{CK}
		\left(q_{\alpha}(t) * q_{\alpha}(s)\right)(x)=\int_{\mathbb{R}^{d}} q_{\alpha}(t, x-y) q_{\alpha}(s, y) \mathrm{d} y=q_{\alpha}(t+s, x), \ \  t, s>0 .
	\end{equation}
		\subsection{Auxiliary lemmas}
		In this part, we recall and extend some lemmas in \cite{wu2023well} and \cite{hao2023} for later use. In concrete terms, the following Lemma \ref{lem:0425} plays an important role in the estimation of $\mathscr{I}_4$ in Section \ref{sec3}. Moreover, when the driven noise is the Brownian motion, the following lemma is similar to \cite[Lemma 7.9]{hao2023}.
		\begin{lemma}\label{lem:0425}
			There is a constant $C=C(d)$ such that for any $f_1, \nabla^k f_2 \in L^{\infty} (\mathbb{R}^d)$ with $k=1,2, h \in (0,1)$ and $s>h,$
			\begin{align} \label{lem2.1}
				&\left| \mathbb{E}f_1(X_{\pi_h(s)}^h)(f_2(X_s^h)-f_2(X_{\pi_h(s)}^h))   \right|   \notag\\
				&\le Ch \|f_1\|_{\infty} \left( \|\nabla f_2 \|_{\infty} \|b\|_{\infty} +  \int_{\mathbb{R}^d} [\|f_2\|_{\infty} \wedge (|y|^2 \| \nabla^2 f_2 \|_{\infty} )] \frac{1}{|y|^{d+\alpha}} \dif y \right).
			\end{align}
		\end{lemma}
		\begin{proof}
			Since $L_s-L_{\pi_h(s)}$ is independent of $X_{\pi_h(s)}^h,$ by the construction \eqref{Euler2}, one sees that 
			\begin{align*}
				\mathscr{I}_h &:=\mathbb{E}f_1(X_{\pi_h(s)}^h) (f_2(X_s^h)-f_2(X_{\pi_h(s)}^h))  \\
				& =\int_{\mathbb{R}^{2d}} f_1(x) \cdot \left(f_2 \left(x+\int_{\pi_h(s)}^{s} b^h(r,x) \dif r +y \right) -f_2(x) \right) \rho_{\pi_h(s)}^h (x) q_{\alpha}(s-\pi_h(s), y) \dif x \dif y.
			\end{align*}
			Then we have
			\begin{align*}
				|\mathscr{I}_h| \le& \|f_1\|_{\infty} \int_{\mathbb{R}^{2d}} \left| f_2 \left( x+\int_{\pi_h(s)}^{s} b^h(r,x) \dif r+y \right) -f_2(x+y)  \right| \rho_{\pi_h(s)}^h (x) q_{\alpha} (s-\pi_h(s), y) \dif x \dif y \\
				& + \left| \int_{\mathbb{R}^{2d}} f_1(x) \cdot \left( f_2(x+y)-f_2(x)\right) \rho_{\pi_h(s)}^h (x) q_{\alpha} (s-\pi_h(s),y) \dif x \dif y \right| \\
				=:&\mathscr{I}_1+\mathscr{I}_2.
			\end{align*}
			For $\mathscr{I}_1$, it is easy to see that 
			\begin{align*}
				\mathscr{I}_1 & \le \|f_1\|_{\infty} \|\nabla f_2\|_{\infty} \int_{\mathbb{R}^{2d}} \int_{\pi_h(s)}^{s} |b^h(r,x)| \dif r \rho_{\pi_h(s)}^h(x) q_{\alpha}(s-\pi_h(s), y) \dif x \dif y \\
				& \le h \|f_1\|_{\infty} \|\nabla f_2\|_{\infty} \|b\|_{\infty} \int_{\mathbb{R}^{2d}} \rho_{\pi_h(s)}^h (x) q_{\alpha} (s-\pi_h(s),y) \dif x \dif y  = h \|f_1\|_{\infty} \|\nabla f_2\|_{\infty} \|b\|_{\infty}.
			\end{align*}
			For $\mathscr{I}_2,$ it  follows from the symmetry of $q_{\alpha}(t, \cdot)$ that
			\begin{align*}
				\mathscr{I}_2 & =\frac{1}{2} \left| \int_{\mathbb{R}^{2d}} f_1(x) \cdot \left( f_2(x+y)+f_2(x-y)-2f_2(x) \right) \rho_{\pi_h(s)}^h(x) q_{\alpha}(s-\pi_h(s),y) \dif x \dif y   \right|,
			\end{align*}
			where 
			\begin{align*}
				&\int_{\mathbb{R}^d} | f_2(x+y)+f_2(x-y)-2f_2(x) | q_{\alpha}(t,y) \dif y \\
				& \le \int_{\mathbb{R}^d} [ \| f_2\|_{\infty} \wedge (|y|^2 \| \nabla^2 f_2 \|_{\infty} )   ] q_{\alpha} (t,y) \dif y \\
				&  \overset{\eqref{0423:00}}{\lesssim}\int_{\mathbb{R}^d} [\| f_2 \|_{\infty} \wedge (|y|^2 \| \nabla^2 f_2 \|_{\infty} ) ] t^{-\frac{d}{\alpha}} \frac{1}{\left(1+| t^{-\frac{1}{\alpha}} y| \right)^{d+\alpha}} \dif y \\
				& \le \int_{\mathbb{R}^d} [\| f_2 \|_{\infty} \wedge (|y|^2 \| \nabla^2 f_2 \|_{\infty})] t^{-\frac{d}{\alpha}} \frac{1}{t^{-\frac{1}{\alpha}(d+\alpha)} |y|^{d+\alpha}} \dif y \\
				& =t \int_{\mathbb{R}^d} [\| f_2 \|_{\infty} \wedge (|y|^2 \| \nabla^2 f_2 \|_{\infty})] \frac{1}{|y|^{d+\alpha}} \dif y.
			\end{align*}
			Therefore,
			\begin{align*}
				\mathscr{I}_2 \lesssim h \|f_1\|_{\infty} \int_{\mathbb{R}^d} [\|f_2\|_{\infty} \wedge (|y|^2 \| \nabla^2 f_2 \|_{\infty} )] \frac{1}{|y|^{d+\alpha}} \dif y.
			\end{align*}
			This completes the proof.
		\end{proof}

		Next, we present the following useful 
		lemma which gives the time H\"older regularity of $\rho^h$.
		\begin{lemma}\label{lem:22}
			Let $\mu_{0}(\dif x)=\mathbb{P} \circ X_0^{-1}(\dif x)$ be the distribution of $X_0$ and $\alpha \in (1,2).$ 
			For any $T >0 $ and $\beta \in (0, \alpha-1],$  there is a constant $c=c(d, \alpha, T, \|b\|_{\infty}, \beta)>0$ such that for all $N \in \mathbb{N}, $ $t_1, t_2 \in (0,T]$ and $y \in \mathbb{R}^d,$
			\begin{equation} \label{alpha-1}
				|\rho^h_{t_2}(y)-\rho^h_{t_1}(y)| \le c |t_2-t_1|^{\beta/\alpha} \sum_{i=1,2} t_i^{-\beta/\alpha} \int_{\mathbb{R}^d} q_{\alpha}(t_i, x-y) \mu_{0}(\dif x).
			\end{equation}
		\end{lemma}
		In light of \eqref{alpha-1} with $\beta=\alpha-1$, by Young's inequality, one sees that if $\mu_0$ admits a density $\rho_0\in L^q(\mR^d)\cap L^1(\mR^d)$, then for any $p \in [1,q]$ and $0 < s <t \le T,$
			\begin{align} \label{10}
				\| \rho_s^h-\rho_t^h\|_{p} 
				& \lesssim  |t-s|^{\frac{\alpha-1}\alpha} s^{-\frac{\alpha-1}\alpha} (\|q_{\alpha}(t)\|_{1} + \| q_{\alpha}(s)\|_{1}) \| \rho_0 \|_{p} \notag \\
				& \lesssim |t-s|^{\frac{\alpha-1}\alpha} s^{-\frac{\alpha-1}\alpha} \| \rho_0 \|_{p}.
			\end{align}
			
		It should be noted that Lemma \ref{lem:22} is proved in \cite[Corollary 4.4]{wu2023well} for $\beta\in(0,\alpha-1)$. Since we need to use the critical case $\beta=\alpha-1$ in next section, we show it in the following.
		To this end, we first introduce the following lemma from \cite[Lemma 2.3]{wu2023well}.
		\begin{lemma}
			For any $\beta \in(0, \alpha)$ and $j \in \mathbb{N}_{0}$, there is a constant $c=c(d, \alpha, \beta, j)>0$ such that for every $t_{1}, t_{2}>0$ and $x \in \mathbb{R}^{d}$,
			\begin{equation} \label{lem2.3}
				\left|\nabla^{j} q_{\alpha}\left(t_{1}, x\right)-\nabla^{j} q_{\alpha}\left(t_{2}, x\right)\right| \leqslant c\left|t_{2}-t_{1}\right|^{\beta / \alpha}\left(t_{1}^{-(j+\beta) / \alpha} q_{\alpha}\left(t_{1}, x\right)+t_{2}^{-(j+\beta) / \alpha} q_{\alpha}\left(t_{2}, x\right)\right) .
			\end{equation}
		\end{lemma}
		We also recall the following Duhamel's formula from \cite[Lemma 4.1]{wu2023well} and \cite[Remark 4.2]{wu2023well}.
		\begin{lemma}(Duhamel's formula)
			Let $\alpha \in (1,2).$ For each $t \in (0,T]$ and $x \in \mathbb{R}^d,$ the density $\rho^h_t$ satisfying the following Duhamel's formula:
			\begin{equation} \label{Duhamel}
				\rho^h_t(y)=\int_{\mR^d}q_{\alpha}(t,x-y)\mu_0(\dif x)+\int_{0}^{t} \mathbb{E} \left[b^h(s,X^h_{\pi_h(s)}) \cdot \nabla q_{\alpha}(t-s,X_s^h-y) \right] \dif s.
			\end{equation}
		\end{lemma}
		The following uniform estimate is from \cite[Theorem 4.3]{wu2023well}: there is a constant $c>0$ depending on $d, \alpha, T,\|b\|_{\infty}$ such that for any $h \in (0,1), t \in(0, T]$ and $x, y \in \mathbb{R}^{d}$,
		\begin{equation} \label{uniform-estimate}
			\rho^h_t(y) \le c \int_{\mathbb{R}^d} q_{\alpha}(t,x-y) \mu_0(\dif x)= c \int_{\mathbb{R}^d} q_{\alpha}(t,x-y) \rho_0(x)\dif x,	\end{equation}
		which by Young's inequality and \eqref{estimate-varro} implies that for any $p,q,r\in[1,\infty]$ with $1+1/r=1/p+1/q$,
		\begin{align}\label{0423:04}
			\|\rho^h_t\|_r\lesssim \|q_{\alpha}(t)\|_p\|\rho_0\|_q\lesssim t^{-\frac{d}{\alpha}(\frac1q-\frac1r)}\|\rho_0\|_q.
		\end{align}
		Now we give the 
		\begin{proof}[Proof of Lemma \ref{lem:22}]
			The case of $\beta\in(0,\alpha-1)$ is shown in \cite[Corollary 4.4]{wu2023well}, here we only show it for $\beta=\alpha-1$.
			Without loss of generality, we suppose that $t_1<t_2.$ In view of Duhamel's formula \eqref{Duhamel}, one sees that
			\begin{align*}
				| \rho^h_{t_2}(y)-\rho^h_{t_1}(y)| \le&\int_{\mathbb{R}^d} | q_{\alpha}(t_2,x-y)-q_{\alpha}(t_1,x-y) | \mu_{0}(\dif x)\\
				&+\int_{0}^{t_1} \mathbb{E} \left|b^h(s,X^h_{\pi_h(s)}) \cdot (\nabla q_{\alpha}(t_2-s)-\nabla q_{\alpha}(t_1-s))(X_s^h-y) \right| \dif s\\
				&+\int_{t_1}^{t_2} \mathbb{E} \left|b^h(s,X^h_{\pi_h(s)}) \cdot \nabla q_{\alpha}(t_2-s,X_s^h-y) \right| \dif s\\
				=:&\mathscr{I}_1+\mathscr{I}_2+ \mathscr{I}_3.
			\end{align*}
			For $\mathscr{I}_1,$ by \eqref{lem2.3} with $\beta=\alpha-1$, we have 
			\begin{equation*}
				\mathscr{I}_1 \lesssim |t_2-t_1|^{(\alpha-1)/\alpha} \sum_{i=1,2} t_i^{-(\alpha-1)/\alpha} \int_{\mathbb{R}^d} q_{\alpha}(t_i,x-y) \mu_{0}(\dif x).
			\end{equation*}
			For $\mathscr{I}_2,$ it follows from taking $\beta=1$ in \eqref{lem2.3} that  
			\begin{align*}
				&\quad| \nabla q_{\alpha} (t_2-s,z-y)-\nabla q_{\alpha}(t_1-s,z-y)|\lesssim |t_2-t_1|^{\frac{1}{\alpha}} (t_1-s)^{-\frac{2}{\alpha}} \sum_{i=1,2} q_{\alpha}(t_i-s,z-y),\end{align*}
			which by \eqref{nabla-p} and \eqref{0423:00} implies that
			\begin{align*}
				&\quad| \nabla q_{\alpha} (t_2-s,z-y)-\nabla q_{\alpha}(t_1-s,z-y)|\\
				&\lesssim \left[ [|t_2-t_1|^{\frac{1}{\alpha}} (t_1-s)^{-\frac{2}{\alpha}} ] \wedge (t_1-s)^{-\frac{1}{\alpha}} \right] \sum_{i=1,2} q_{\alpha}(t_i-s,z-y).
			\end{align*}
			Thus, we have
			\begin{align*}
				\mathscr{I}_2 &\le\|b\|_{\infty} \int_{0}^{t_1} \int_{\mathbb{R}^d} | \nabla q_{\alpha} (t_2-s,z-y)-\nabla q_{\alpha}(t_1-s,z-y)| \rho^h_s(z) \dif z \dif s\\
				&\lesssim \int_{0}^{t_1} \left[ [|t_2-t_1|^{\frac1\alpha} (t_1-s)^{-\frac{2}{\alpha}} ] \wedge (t_1-s)^{-\frac{1}{\alpha}} \right] \sum_{i=1,2} \int_{\mathbb{R}^d} q_{\alpha}(t_i-s,z-y) \rho^h_s(z)\dif z\dif s. 
			\end{align*}			
			Noting that for $i=1,2$,
			\begin{align*}
				&\int_{\mathbb{R}^d} q_{\alpha}(t_i-s,z-y) \rho^h_s(z)\dif z\\
				\overset{\eqref{uniform-estimate}}{\lesssim} &\int_{\mathbb{R}^d}\int_{\mathbb{R}^d} q_{\alpha}(t_i-s,z-y)q_{\alpha}(s,x-z)\dif z\mu_{0}(\dif x) \dif s \\
				\overset{\eqref{CK}}{\lesssim} &\int_{\mathbb{R}^d}q_{\alpha}(t_i,x-y) \mu_{0}(\dif x),
			\end{align*}
			and for $\alpha \in (1,2)$,  
			\begin{align*}
				&\quad\int_{0}^{t_1} \left[ [|t_2-t_1|^{\frac1\alpha} (t_1-s)^{-\frac{2}{\alpha}} ] \wedge (t_1-s)^{-\frac{1}{\alpha}} \right]\dif s=\int_{0}^{t_1} \left[ [|t_2-t_1|^{\frac1\alpha} s^{-\frac{2}{\alpha}} ] \wedge s^{-\frac{1}{\alpha}} \right]\dif s\\
				&=\int_{0}^{t_1} s^{-\frac{1}{\alpha}}\left[ [|t_2-t_1|^{\frac1\alpha} s^{-\frac{1}{\alpha}} ] \wedge 1 \right]\dif s\le|t_2-t_1|^{\frac{\alpha-1}{\alpha}}\int_{0}^{\infty} r^{-\frac{1}{\alpha}}[r^{-\frac{1}{\alpha}}\wedge 1]\dif r,
			\end{align*}
			where the last inequality is provided by a change of variable $s=(t_2-t_1)r$. Therefore, we have
			\begin{equation*}
				\mathscr{I}_2 \lesssim |t_1-t_2|^{\frac{\alpha-1}{\alpha}} \sum_{i=1,2} \int_{\mathbb{R}^d} q_{\alpha}(t_i,x-y) \mu_{0}(\dif x).
			\end{equation*}
			For $\mathscr{I}_3,$ by \eqref{uniform-estimate}, \eqref{nabla-p} and \eqref{CK}, we have
			\begin{align*}
				\mathscr{I}_3 &\le\|b\|_{\infty} \int_{t_1}^{t_2} \mE| \nabla q_{\alpha}(t_2-s,X^h_{s}-y)| \dif s\\	
				&=\|b\|_{\infty} \int_{t_1}^{t_2} \int_{\mathbb{R}^d} | \nabla q_{\alpha}(t_2-s,z-y)| \rho^h_s(z) \dif z\dif s\\
				& \lesssim \int_{t_1}^{t_2} (t_2-s)^{-1/\alpha} \left( \int_{\mathbb{R}^d}  q_{\alpha} (t_2-s,z-y) \int_{\mathbb{R}^d} q_{\alpha} (s,x-z) \mu_{0}(\dif x) \dif z \right) \dif s \\
				& =\int_{t_1}^{t_2} (t_2-s)^{-1/\alpha} \dif s \int_{\mathbb{R}^d} q_{\alpha}(t_2,x-y) \mu_{0} (\dif x) \\
				& \lesssim (t_2-t_1)^{-1/\alpha+1} \int_{\mathbb{R}^d} q_{\alpha}(t_2,x-y) \mu_{0}(\dif x).
			\end{align*}
			Combining the above calculations, we get the desired estimate.
		\end{proof}
		
		\section{Proof of the main result}\label{sec3}
		\begin{proof}[Proof of Theorem \ref{11}]
			We recall
			\begin{equation*}
				b^0(s,x):=b(s,x,\rho_s(x)).
			\end{equation*}
			For any $\phi \in C_b^{\infty}$ and $t \in [0,T],$ we let $u^t(s,x)=q_{\alpha}(t-s) \ast \phi(x)$. Then based on \eqref{0423:02}, $u^t$ solves the following backward heat equation
			\begin{equation} \label{4}
				\partial_s u^t+ \Delta^{\frac{\alpha}{2}} u^t=0, \quad u^t(t,x)=\phi(x), \quad s \in (0,t), \ x \in \mathbb{R}^d.
			\end{equation}
			By It\^o's formula to $u^t(s,X_s)$ and $u^t(s,X_s^h),$ we have
			\begin{align*}
				&\quad u^t(t,X_t)-u^t(0,X_0) \\
				&=\int_{0}^{t} (\partial_s u^t+b^0\cdot\nabla u^t+\Delta^{\frac\alpha2}u^t)(s,X_{s}) \dif s \\
				& \qquad +\int_{0}^{t} \int_{\mR^d} \left(u^t(s,X_{s-}+z)-u^t(s,X_{s-}) \right)\widetilde{N}(\dif s, \dif z),
			\end{align*}
			and
			\begin{align*}
				&\quad u^t(t,X_t^h)-u^t(0,X_0) \\
				&=\int_{0}^{t} (\partial_s u^t+\Delta^{\frac\alpha2}u^t)(s,X_{s}^h)+b^h(s,X^h_{\pi_h(s)})\cdot\nabla u^t(s,X_s^h) \dif s \\
				&\qquad +\int_{0}^{t} \int_{\mR^d} \left(u^t(s,X_{s-}^h+z)-u^t(s,X_{s-}^h) \right)\widetilde{N}(\dif s, \dif z),
			\end{align*}
			which implies that 
			\begin{equation*}
				\mathbb{E}u^t(t,X_t)=\mathbb{E}u^t(0,X_0)+\mathbb{E} \int_{0}^{t} \left( b^0 \cdot \nabla u^t \right)(s,X_{s}) \dif s,
			\end{equation*}
			and
			\begin{equation*}
				\mathbb{E}u^t(t,X_t^h)=\mathbb{E}u^t(0,X_0)+\mathbb{E} \int_{0}^{t} b^h(s,X^h_{\pi_h(s)}) \cdot \nabla u^t(s,X_{s}) \dif s.
			\end{equation*}
			Therefore,
			\begin{align*}
				\mathbb{E} \phi(X_t^h) &=\mathbb{E} u^t(0,X_0)+\mathbb{E} \int_{0}^{t} b^h(s,X^h_{\pi_h(s)}) \cdot \nabla u^t(s, X^h_{s}) \dif s \\
				&=\mathbb{E} \phi(X_t)+\mathbb{E} \int_{0}^{t} b^h(s, X_{\pi_h(s)}^h) \cdot \nabla u^t(s, X_{s}^h) \dif s-\mathbb{E} \int_{0}^{t} (b^0\cdot \nabla u^t)(s,X_{s}) \dif s,
			\end{align*}
			which implies that
			\begin{align*}
				|\mathbb{E} \phi(X_t^h)-\mathbb{E} \phi(X_t)| 	\le& \bigg| \mathbb{E} \int_{0}^{t} \left( (b^0 \cdot \nabla u^t)(s,X_{s}^h)-(b^0 \cdot \nabla u^t)(s,X_{s})  \right) \dif s \bigg| \\
				&+ \bigg| \mathbb{E} \int_{0}^{t} \left( b(s,X_{s}^h, \rho_{s}^h(X_{s}^h))-b^0(s,X_{s}^h) \right) \cdot \nabla u^t(s,X_{s}^h) \dif s \bigg| \\
				&+\bigg| \mathbb{E} \int_{0}^{t} \left(b^h(s,X_{s}^h)-b \left(s,X_{s}^h, \rho^h_{s}(X_{s}^h)  \right)  \right) \cdot \nabla u^t(s,X_{s}^h) \dif s \bigg| \\
				&+\bigg| \mathbb{E} \int_{0}^{t} \left(b^h(s,X_{\pi_h(s)}^h) -b^h( s, X_{s}^h) \right) \cdot \nabla u^t(s,X_{s}^h) \dif s  \bigg| \\
				=:&\mathscr{I}_1+\mathscr{I}_2+\mathscr{I}_3+\mathscr{I}_4.
			\end{align*}
			Before calculating these four terms, we note that
			by \eqref{estimate-varro}, for $k\in\mN_0$,
			\begin{align} \label{7}
				\| \nabla^k u^t(s) \|_{\infty} \lesssim (t-s)^{-\frac{k}{\alpha}} \|\phi\|_{\infty},
			\end{align}
			from which we conclude that
			\begin{align}
				\mathscr{I}_1 & =\left| \int_{0}^{t} \left( (b^0 \cdot \nabla u^t)(s,x) \rho_{s}^h(x) -(b^0 \cdot \nabla u^t) (s,x) \cdot \rho_{s}(x) \right) \dif s   \right| \no\\
				& \le \int_{0}^{t} \|\rho_{s}^h-\rho_{s}\|_{1} \|b^0 \cdot \nabla u^t(s)\|_{\infty} \dif s \no\\
				& \lesssim \|\phi\|_{\infty} \|b\|_{\infty} \int_{0}^{t} (t-s)^{-\frac{1}{\alpha}} \cdot \|\rho_{s}^h-\rho_{s}\|_{1} \dif s.\label{I1}
			\end{align}
			For $\mathscr{I}_2,$ taking $r=\infty$ in \eqref{0423:04}, we have
			\begin{align*} 
				\|\rho_t^h\|_\infty \lesssim t^{-\frac{d}{q \alpha}} \| \rho_{0}\|_q,
			\end{align*}
			which by \eqref{6} and \eqref{7} implies that
			\begin{align}
				\mathscr{I}_2 & =\left| \mathbb{E} \int_{0}^{t} \left( b(s, X_{s}^h, \rho_s^h(X_{s}^h)) -b(s,X_{s}^h, \rho_{s}(X_{s}^h)) \right) \cdot \nabla u^t(s,X_{s}^h)  \dif s \right| \no\\
				& \le \kappa\mathbb{E} \int_{0}^{t}|\rho_{s}^h(X_{s}^h)-\rho_{s}(X_{s}^h)| \cdot |\nabla u^t(s,X_{s}^h)| \dif s\no \\
				& =\kappa\int_{0}^{t} \int_{\mathbb{R}^d} | \rho_{s}^h(x)-\rho_{s}(x)| \cdot |\nabla u^t(s,x)| \rho_{s}^h(x)  \dif x \dif s\no \\
				& \lesssim \int_{0}^{t} \int_{\mathbb{R}^d} | \rho_{s}^h(x)-\rho_{s}(x)| (t-s)^{-\frac{1}{\alpha}} \|\phi\|_{\infty} \rho_{s}^h(x) \dif x \dif s \no\\
				& \lesssim \|\phi\|_{\infty}  \int_{0}^{t} (t-s)^{-\frac{1}{\alpha}} s^{-\frac{d}{q\alpha}}\| \rho_{s}^h(x)-\rho_{s}(x) \|_{1} \dif s.\label{I2}
			\end{align}
			For $\mathscr{I}_3,$  we note that
			\begin{align*}
				\mathscr{I}_3 & =\bigg| -\mathbb{E} \int_{0}^{h} b \left( s,X^h_{s} , \rho^h_{s} (X^h_{s})\right)  \cdot \nabla u^t(s,X^h_{s}) \dif s \\
				& \qquad + \mathbb{E} \int_{h}^{t} \left( b (s, X^h_{s} , \rho_{\pi_h(s)}^h(X^h_{s}))-b \left( s,X^h_{s} , \rho^h_{s} (X^h_{s})\right)  \right) \cdot \nabla u^t(s,X^h_{s}) \dif s  \bigg| \\
				& \overset{\eqref{6}}{\le}\|b\|_\infty\int_{0}^{h} \|\nabla u^t(s)\|_\infty \dif s + \kappa\mathbb{E} \int_{h}^{t} |\rho^h_{\pi_h(s)}(X^h_{s})-\rho^h_{s}(X^h_{s})| |\nabla u^t(s,X^h_{s})| \dif s \\
				&=\|b\|_\infty\int_{0}^{h} \|\nabla u^t(s)\|_\infty \dif s + \kappa\int_{h}^{t}\int_{\mR^d} |\rho^h_{\pi_h(s)}(x)-\rho^h_{s}(x)| |\nabla u^t(s,x)|\rho^h_s(x)\dif x \dif s,
			\end{align*}
			where in view of \eqref{10} with $p=2$ and \eqref{0423:04} with $r=q=2$, by H\"older's inequality and \eqref{7}, one sees that
			\begin{align*}
				&\quad\int_{\mR^d} |\rho^h_{\pi_h(s)}(x)-\rho^h_{s}(x)| |\nabla u^t(s,x)|\rho^h_s(x)\dif x\\
				&\le \|\rho^h_{\pi_h(s)}-\rho^h_{s}\|_2\|\nabla u^t(s)\|_\infty\|\rho^h_s\|_2\\
				&\lesssim (s-\pi_h(s))^{\frac{\alpha-1}\alpha}(\pi_h(s))^{-\frac{\alpha-1}{\alpha}}(t-s)^{-\frac1\alpha}\|\phi\|_\infty\|\rho_0\|_2^2.
			\end{align*}
			Since $q\ge2$ and $\rho_0\in L^q\cap L^1\subset L^2$, we have
			\begin{align}
				\mathscr{I}_3&\lesssim \int_{0}^{h} \|\nabla u^t(s)\|_\infty \dif s +\|\phi\|_\infty\int_{h}^{t}(s-\pi_h(s))^{\frac{\alpha-1}\alpha}(\pi_h(s))^{-\frac{\alpha-1}{\alpha}}(t-s)^{-\frac1\alpha} \dif s\no\\
				&\overset{\eqref{7}}{\lesssim}\|\phi\|_\infty\left(\int_{0}^{h}(t-s)^{-\frac{1}{\alpha}}\dif s+ h^{\frac{\alpha-1}\alpha}\int_{h}^{t}(\pi_h(s))^{-\frac{\alpha-1}{\alpha}}(t-s)^{-\frac1\alpha} \dif s\right)\no\\
				&\lesssim \|\phi\|_\infty\left(\int_{0}^{h}(h-s)^{-\frac{1}{\alpha}}\dif s+ h^{\frac{\alpha-1}\alpha}\int_{h}^{t}(s-h)^{-\frac{\alpha-1}{\alpha}}(t-s)^{-\frac1\alpha} \dif s\right)\lesssim \|\phi\|_\infty h^{\frac{\alpha-1}\alpha}.\label{I3}
			\end{align}
			For $\mathscr{I}_4,$ without loss of generality, we assume $t>2h$. Otherwise, it follows from \eqref{7} that
			\begin{align*}
				\mathscr{I}_4\le 2\|b\|_\infty\int_0^{t}\|\nabla u^t(s)\|_\infty\dif s\lesssim\|\phi\|_\infty \int_0^{t}(t-s)^{-\frac{1}{\alpha}}\dif s\lesssim t^{\frac{\alpha-1}{\alpha}}\|\phi\|_\infty\lesssim h^{\frac{\alpha-1}{\alpha}}\|\phi\|_\infty.
			\end{align*}
			 Then we make the following decomposition
			\begin{align*}
				\mathscr{I}_4 
				& \le \bigg| \mathbb{E} \int_{0}^{t} \left( (b^h \cdot \nabla u^t) (s,X^h_{{\pi_h}(s)})- (b^h \cdot \nabla u^t) (s,X_{s}^h)   \right) \dif s \bigg| \\
				& \qquad + \bigg| \mathbb{E} \int_{0}^{t} b^h(s,X^h_{\pi_h(s)}) \cdot \left( \nabla u^t(s, X^h_{s})-\nabla u^t(s,X^h_{\pi_{h(s)}}) \right) \dif s \bigg| \\
				&=: \mathscr{I}_{41}+\mathscr{I}_{42},
			\end{align*}
			where 
			\begin{equation*}
				\mathscr{I}_{41}=\bigg| \int_{h}^{t} \int_{\mathbb{R}^d} (b^h \cdot \nabla u^t)(s,x)(\rho^h_{\pi_{h(s)}}(x)-\rho_{s}^h(x)) \dif x \dif s \bigg|.
			\end{equation*}
			Then by \eqref{7} and \eqref{10}, we have 
			\begin{align}
				\mathscr{I}_{41} &\le \int_{h}^{t} \| b^h \cdot \nabla u^t(s) \|_{\infty} \| \rho^h_{\pi_h(s)}- \rho_{s}^h\|_{1} \dif s\no \\
				& \lesssim \|b\|_{\infty} \| \phi\|_{\infty} \int_{h}^{t} (t-s)^{-\frac{1}{\alpha}} (s-\pi_{h(s)})^{\frac{\alpha-1}{\alpha}} (\pi_h(s))^{-\frac{\alpha-1}{\alpha}}\dif s\no\\
				& \lesssim h^{\frac{\alpha-1}{\alpha}} \| \phi\|_{\infty} \int_{0}^{t} (t-s)^{-\frac{1}{\alpha}} (s-h)^{-\frac{\alpha-1}{\alpha}} \dif s \lesssim h^{\frac{\alpha-1}{\alpha}} \| \phi\|_{\infty}.\label{0423:03}
			\end{align}
			For $\mathscr{I}_{42},$ by \eqref{lem2.1}, 
			we have
			\begin{align*}
				\mathscr{I}_{42} & = \bigg| \mathbb{E} \int_{h}^{t} b^h(s, X^h_{\pi_{h(s)}}) \cdot \left( \nabla u^t(s,X^h_{s}) - \nabla u^t(s,X^h_{\pi_{h(s)}})  \right) \dif s \bigg| \\
				&\le\bigg|\mathbb{E} \int_{h}^{t-h} b^h(s, X^h_{\pi_{h}(s)}) \left(\nabla u^t(s,X_s^h)-\nabla u^t(s,X^h_{\pi_h(s)}) \right) \dif s\bigg|  \\
				& \qquad +\bigg|\mathbb{E} \int_{t-h}^{t} b^h(s, X^h_{\pi_{h}(s)}) \left(\nabla u^t(s,X_s^h)-\nabla u^t(s,X^h_{\pi_h(s)}) \right) \dif s   \bigg| \\
				& \overset{\eqref{lem2.1}}{\lesssim} h \|b\|_{\infty} \int_{h}^{t-h} \left( \| \nabla^2 u^t(s) \|_{\infty} \|b\|_{\infty} + \int_{\mathbb{R}^d} [\| \nabla u^t \|_{\infty} \wedge (|y|^2 \|\nabla^3 u^t \|_{\infty}) ]\frac{1}{|y|^{d+\alpha}} \dif y\right) \dif s \\
				& \qquad +\|b\|_{\infty} \int_{t-h}^{t} \| \nabla u^t(s) \|_{\infty} \dif s, 
			\end{align*}
			where from \eqref{7} and a change of variable, it follows that
			\begin{align*}
				&\quad\| \nabla^2 u^t(s) \|_{\infty} \|b\|_{\infty} + \int_{\mathbb{R}^d} [\| \nabla u^t \|_{\infty} \wedge (|y|^2 \|\nabla^3 u^t \|_{\infty}) ]\frac{1}{|y|^{d+\alpha}} \dif y\\
				&\lesssim\|\phi\|_\infty\left((t-s)^{-\frac{2}{\alpha}}+\int_{\mathbb{R}^d} [(t-s)^{-\frac1\alpha}\wedge (|y|^2(t-s)^{-\frac3\alpha}) ]\frac{1}{|y|^{d+\alpha}} \dif y\right)\\
				&\lesssim\|\phi\|_\infty\left((t-s)^{-\frac{2}{\alpha}}+(t-s)^{-\frac{1+\alpha}{\alpha}}\int_{\mathbb{R}^d} [1\wedge |y|^2]\frac{1}{|y|^{d+\alpha}} \dif y\right)\lesssim \|\phi\|_\infty(t-s)^{-\frac{1+\alpha}{\alpha}}.
			\end{align*}
			Thus, we get
			\begin{align*}
				\mathscr{I}_{42}&\lesssim h\|\phi\|_\infty\int_{h}^{t-h}(t-s)^{-\frac{1+\alpha}{\alpha}}\dif s+\int_{t-h}^{t} \| \nabla u^t(s) \|_{\infty} \dif s\\
				&\overset{\eqref{7}}{\lesssim} h\|\phi\|_\infty\int_{h}^{t-h}(t-s)^{-\frac{1+\alpha}{\alpha}}\dif s+\|\phi\|_\infty\int_{t-h}^{t} (t-s)^{-\frac1\alpha}\dif s\lesssim \|\phi\|_\infty h^{\frac{\alpha-1}\alpha},
			\end{align*}
			with which and \eqref{0423:03} we have
			$
			\mathscr{I}_{4}\lesssim h^{\frac{\alpha-1}\alpha}\|\phi\|_\infty .
			$
			
			In summary, by taking the supermum of $\|\phi\|_\infty=1$ in \eqref{I1}, \eqref{I2}, \eqref{I3} and $\mathscr{I}_{4}$, we have
			\begin{equation*}
				\|\rho_t-\rho_t^h\|_{1} \lesssim \int_{0}^{t} (t-s)^{-\frac{1}{\alpha}} s^{-\frac{d}{q \alpha}} \|\rho_s-\rho_s^h\|_{1} \dif s+h^{\frac{\alpha-1}{\alpha}}.
			\end{equation*}
			By Gronwall inequalities of Volterra's type in \cite[Lemma 2.2]{zhang2010stochastic} (see also \cite[Lemma A.4]{hao2023}), we have
			\begin{align*}
				\| \rho_t-\rho_t^h \|_{1} \lesssim h^{\frac{\alpha-1}{\alpha}}.
			\end{align*}
			From this, we complete the proof.
		\end{proof}
\begin{remark}
\rm 
In the case where the function $b(t,x,u)$ simplifies to $b(t,x)$, the term $\sI_2$ is omitted in the proof, and $\sI_3$ reduces to $|\mE\int_0^h b\cdot\nabla u^t(s,X^h_s)\dif s| \lesssim h^{(\alpha-1)/\alpha}$. Notably, the estimation of $\sI_4$ only requires $\mu_0$ to be a finite measure, thus ensuring the convergence rate akin to \eqref{0425:00} without any assumption on $\mu_0$.
\end{remark}

\subsection*{Acknowledgments}
We are deeply grateful to Prof. Rongchan Zhu for her valuable suggestions and for correcting some errors. 


\bibliographystyle{amsplain}
\bibliographystyle{amsplain}

\end{document}